\documentclass{siamltex} 

\usepackage[T1]{fontenc}
\usepackage[latin1]{inputenc}
\usepackage[english]{babel}
\usepackage{graphicx}
\usepackage{url}
\usepackage{graphics}
\usepackage{epsfig,color}
\usepackage{amsmath,amssymb}
 \usepackage{relsize}
\renewcommand{\theequation}{\thesection.\arabic{equation}}

\newtheorem{thm}{Theorem}[section]

\newtheorem{lem}[thm]{Lemma}

\newtheorem{rem}[thm]{Remark}

\begin{document}
\newcommand{\BX}{{\bf X}}
\newcommand{\cv}{{\cal V}}
\newcommand{\cW}{{\cal W}}
\newcommand{\co}{{\cal O}}

\renewcommand{\theequation}{\thesection.\arabic{equation}}
\def\@eqnnum{{\reset@font\rm (\theequation)}}

\def\abstract{
\advance \rightskip by 10mm
\advance \leftskip by 10mm
\vspace{-0.8em}
\noindent
\small{\bf Abstract.}
}
\def\endabstract{\par\normalsize\rm}

\def\Xint#1{\mathchoice
{\XXint\displaystyle\textstyle{#1}}%
{\XXint\textstyle\scriptstyle{#1}}%
{\XXint\scriptstyle\scriptscriptstyle{#1}}%
{\XXint\scriptscriptstyle\scriptscriptstyle{#1}}%
\!\int}
\def\XXint#1#2#3{{\setbox0=\hbox{$#1{#2#3}{\int}$}
\vcenter{\hbox{$#2#3$}}\kern-.5\wd0}}
\def\ddashint{\Xint=}
\def\dashint{\Xint-}

\def\a{\alpha}
\def\b{\beta}
\def\d{\delta}\def\D{\Delta}
\def\e{\epsilon}
\def\g{\gamma}\def\G{\Gamma}
\def\k{\kappa}
\def\lam{\lambda}\def\Lam{\Lambda}
\renewcommand\o{\omega}\renewcommand\O{\Omega}
\def\s{\sigma}\def\S{\Sigma}
\renewcommand\t{\theta}\def\vt{\vartheta}
\newcommand{\vphi}{\varphi}
\def\z{\zeta}

\newcommand{\tsigma}{\tilde{\s}}
\newcommand{\tbsigma}{\tilde{\bsigma}}
\def\te{\tilde{\e}}
\def\tu{\tilde{u}}

\newcommand{\bchi}{\mbox{\boldmath$\chi$}}
\newcommand{\bdelta}{\mbox{\boldmath$\delta$}}
\newcommand{\bepsilon}{\mbox{\boldmath$\epsilon$}}
\newcommand{\bfeta}{\mbox{\boldmath$\eta$}}
\newcommand{\bgamma}{\mbox{\boldmath$\gamma$}}
\newcommand{\bomega}{\mbox{\boldmath$\omega$}}
\newcommand{\bvphi}{\mbox{\boldmath$\varphi$}}
\newcommand{\bphi}{\mbox{\boldmath$\phi$}}
\newcommand{\bPhi}{\mbox{\boldmath$\Phi$}}
\newcommand{\bpsi}{\mbox{\boldmath$\psi$}}
\newcommand{\bPsi}{\mbox{\boldmath$\Psi$}}
\newcommand{\bsigma}{\mbox{\boldmath$\sigma$}}
\newcommand{\btau}{\mbox{\boldmath$\tau$}}
\newcommand{\bxi}{\mbox{\boldmath$\xi$}}
\newcommand{\brho}{\mbox{\boldmath$\rho$}}
\newcommand{\bbeta}{\mbox{\boldmath$\beta$}}
\newcommand{\bzeta}{\mbox{\boldmath$\zeta$}}

\def\bk{\boldsymbol{\kappa}}
\def\bmu{\boldsymbol\mu}
\def\bxi{\boldsymbol{\xi}}
\def\bz{\boldsymbol{\zeta}}

\def\ba{{\bf a}}
\def\bb{{\bf b}}
\def\bc{{\bf c}}
\def\be{{\bf e}}
\def\bff{{\bf f}}
\def\bg{{\bf g}}
\def\bn{{\bf n}}
\def\bp{{\bf p}}
\def\bq{{\bf q}}
\def\bs{{\bf s}}
\def\bt{{\bf t}}
\def\bu{{\bf u}}
\def\bv{{\bf v}}
\def\bw{{\bf w}}
\def\bx{{\bf x}}
\def\by{{\bf y}}
\def\bzz{{\bf z}}

\def\bD{{\bf D}}
\def\bE{{\bf E}}
\def\bF{{\bf F}}
\def\bH{{\bf H}}
\def\bJ{{\bf J}}
\def\bV{{\bf V}}
\def\bU{{\bf U}}
\def\bW{{\bf W}}
\def\bX{{\bf X}}
\def\bY{{\bf Y}}

\def\cA{{\cal A}}
\def\cC{{\cal C}}
\def\cD{{\cal D}}
\def\cE{{\cal E}}
\def\cF{{\cal F}}
\def\cG{{\cal G}}
\def\cI{{\cal I}}
\def\cJ{{\cal J}}
\def\cK{{\cal K}}
\def\cL{{\cal L}}
\def\cO{{\cal O}}
\def\cP{{\cal P}}
\def\cQ{{\cal Q}}
\def\cR{{\cal R}}
\def\cS{{\cal \Sigma}}
\def\cT{{\cal T}}
\def\cU{{\cal U}}
\def\cV{{\cal V}}

\def\scT{{_\cT}}
\def\sD{{_D}}
\def\sE{{_E}}
\def\sF{{_F}}
\def\sFz{{_{F_z}}}
\def\sK{{_K}}
\def\sI{{_I}}
\def\sb{{_b}}
\def\sN{{_N}}

\def\curl{{{\bf curl} \ }}
\def\rot{{\mbox{rot}\ }}
\def\BPI{{\bf \Pi}}

\def\cth{\cT_h}
\def\ctH{\cT_H}

\def\tJ{\tilde{\J}}

\def\hK{\widehat{K}}
\def\hx{\widehat{x}}
\def\hy{\widehat{y}}
\def\bhv{\widehat{\bv}}

\def\l{\ell}
\def\bl{\boldsymbol{\ell}}
\def\col{\colon}
\def\f12{\frac12}
\def\dfrac{\displaystyle\frac}
\def\dint{\displaystyle\int}
\def\nab{\nabla}
\def\p{\partial}
\def\sm{\setminus}
\def\dsum{\displaystyle\sum}
\newcommand{\pp}[2]{\frac{\partial {#1}}{\partial {#2}}}
\def\bzero{{\bf 0}}

\def\divv{\nab\cdot}
\def\divx{\nab_x\cdot}
\def\divtx{\nab_{t,x}\cdot}
\def\nabx{\nab_x}

\newcommand{\curlt}{{\nabla \times}}
\newcommand{\gperp}{\nabla^{\perp}}
\newcommand{\gradt}{\nabla\cdot}

\def\forallqq{\quad\forall\,}
\def\aph{A^{1/2}}
\def\amh{A^{-1/2}}

\def\osc{{\rm osc \, }}

\def\Im{{\rm Im}}
\newcommand{\tr}{{\rm tr}}
\def\divvr{{\rm div}}
\def\curllr{{\rm curl}}
\def\curll{{\rm curl}}
\def\curl{{\bf curl}}
\newcommand{\bgrad}{{\bf grad}}
\newcommand\diam{\mathrm{diam\,}}
\renewcommand\Im{\mathrm{Im\,}}
\def\Span{\mbox{Span}}
\def\supp{\mbox{supp\,}}
\newcommand{\trace}{{\rm trace}}

\newcommand{\tri}{|\!|\!|}
\newcommand{\ljump}{\lbrack\!\lbrack}
\newcommand{\rjump}{\rbrack\!\rbrack}
\newcommand{\bdm}{\begin{displaymath}}
\newcommand{\edm}{\end{displaymath}}
\newcommand{\beq}{\begin{equation}}
\newcommand{\eeq}{\end{equation}}
\newcommand{\beqa}{\begin{eqnarray}}
\newcommand{\eeqa}{\end{eqnarray}}
\newcommand{\beqas}{\begin{eqnarray*}}
\newcommand{\eeqas}{\end{eqnarray*}}
\newcommand{\ul}{\underline}
\newcommand{\wh}{\widehat}
\newcommand{\la}{\langle}
\newcommand{\ra}{\rangle}

\newcommand{\Lt}{L^2(\Omega)}
\newcommand{\Lts}{L^2(\Omega)^2}
\newcommand{\Ltc}{L^2(\Omega)^3}
\newcommand{\Ho}{H^1(\Omega)}
\newcommand{\Hoh}{H^1(\wh{\Omega})}
\newcommand{\Hoi}{H^1(\Omega_i)}
\newcommand{\Hos}{H^1(\Omega)^2}
\newcommand{\Hoc}{H^1(\Omega)^3}
\newcommand{\Hoch}{H^1(\wh{\Omega})^3}
\newcommand{\Hoci}{H^1(\Omega_i)^3}
\newcommand{\Hoz}{H^1_0(\Omega)}
\newcommand{\Ht}{H^2(\Omega)}
\newcommand{\Hti}{H^2(\Omega_i)}
\newcommand{\Hts}{H^2(\Omega)^2}
\newcommand{\Htc}{H^2(\Omega)^3}
\newcommand{\Htz}{H^0(\Omega)}
\newcommand{\Hh}{H^{1/2}(\Gamma)}
\newcommand{\Hhi}{H^{1/2}(\Gamma_i)}
\newcommand{\Hmh}{H^{-1/2}(\Gamma)}
\newcommand{\Hdiv}{H(\divvr;\,\Omega)}
\newcommand{\Hdivh}{H(\divv;\,\wh \Omega)}
\newcommand{\hcurl}{H(\curl\,A;\,\Omega)}
\newcommand{\Hcurl}{H(\curll\,A;\,\Omega)}
\newcommand{\Hcrl}{H(\curll\,;\,\Omega)}
\newcommand{\hcrl}{H(\curl\,;\,\Omega)}
\newcommand{\Hcrlh}{H(\curll\,;\,\wh\Omega)}
\newcommand{\hcrlh}{H(\curl\,;\,\wh\Omega)}
\newcommand{\Wdiv}{\BW_0(\mbox{\divv}\,;\,\Omega)}
\newcommand{\Wcurl}{\BW_0(\mbox{\curl}\,A;\,\Omega)}
\newcommand{\WcrossV}{\BW \times V}

\def\grad{\nabla}
\def\calS{{\cal S}}
\def\cH{{\cal H}}
\def\ba{{\mathbf{a}}}
\def\cN{{\cal N}}

\def\bE{{\bf E}}
\def\bS{{\bf S}}
\def\br{{\bf r}}
\def\bW{{\bf W}}
\def\bLambda{{\bf \Lambda}}

\def\zT{{z_{_{\cT}}}}
\def\vT{{v_{_{\cT}}}}
\def\uT{{u_{_{\cT}}}}

\newcommand{\dd}{\underline{{\mathbf d}}}
\newcommand{\C}{\rm I\kern-.5emC}
\newcommand{\R}{\rm I\kern-.19emR}
\newcommand{\W}{{\mathbf W}}
\def\3bar{{|\hspace{-.02in}|\hspace{-.02in}|}}
\newcommand{\A}{{\mathcal A}}

\newcommand{\aA}{{ \a_{\sF,_A}}}

\newcommand{\aH}{{ \a_{\sF,_H}}}

\newcommand{\lJump}{[\![}
\newcommand{\rJump}{]\!]}
\newcommand{\jump}[1]{[\![ #1]\!]}

\newcommand{\red}[1]{{\color{red} {#1} }}

\def\Xint#1{\mathchoice
{\XXint\displaystyle\textstyle{#1}}%
{\XXint\textstyle\scriptstyle{#1}}%
{\XXint\scriptstyle\scriptscriptstyle{#1}}%
{\XXint\scriptscriptstyle\scriptscriptstyle{#1}}%
\!\int}
\def\XXint#1#2#3{{\setbox0=\hbox{$#1{#2#3}{\int}$}
\vcenter{\hbox{$#2#3$}}\kern-.5\wd0}}
\def\ddashint{\Xint=}
\def\dashint{\Xint-}

\title{Finite Element Methods for Interface Problems: \\ 
Robust Residual A Posteriori Error Estimates \\
for Discontinuous Approximations
}
\author{
Zhiqiang Cai\thanks{
Department of Mathematics, Purdue University, 150 N. University
Street, West Lafayette, IN 47907-2067, \{caiz, he75\}@purdue.edu.
This work was supported in part by the National Science Foundation
under grants DMS-1217081 and DMS-1522707.}
\and Cuiyu He\footnotemark[2]
 \and Shun Zhang\thanks{Department of Mathematics, 
City University of Hong Kong, Hong Kong, shun.zhang@cityu.edu.hk.
This work was supported in part
by Hong Kong Research Grants Council under the GRF Grant Project No. 11303914, CityU 9042090.}
}
\maketitle

\begin{abstract}
For elliptic interface problems, 
this paper studies residual-based a posteriori error estimations for discontinuous finite element approximations.
For the Crouzeix-Raviart
nonconforming and the discontinuous Galerkin elements in both two- and three-dimensions, 
the global reliability bounds are established with constants independent of the jump of the diffusion coefficient.
Moreover, we obtain these estimates with no assumption on the distribution of the diffusion coefficient.
\end{abstract}

\section{Introduction}\label{intro}
\setcounter{equation}{0}

Let $\O$ be a bounded, open, connected subset in $\R^d$ ($d=2$ or $3$) with a Lipschitz continuous boundary $\partial \O$. 
This paper studies a posteriori error estimates of nonconforming and discontinuous finite element methods for the following  
interface problems (i.e., diffusion problems with discontinuous coefficients):
 \begin{equation}\label{scalar}
 -\nabla\cdot \,(\a(x)\nabla\, u) = f
 \quad \mbox{in} \,\,\Omega
 \end{equation}
with homogeneous Dirichlet boundary condition, $u = 0$ on $\p\O$, for simplicity.
Here, $f \in L^{2}(\O)$ is a given function; 
and diffusion coefficient $\a(x)$ is positive 
and piecewise constant on polygonal subdomains of $\O$ with possible
large jumps across subdomain boundaries (interfaces):
\[
 \a(x)=\a_i > 0\quad\mbox{in }\,\O_i
 \quad\mbox{for }\, i=1,\,...,\,n.
\]
Here, $\{\Omega_i\}_{i=1}^n$ is a partition of the domain $\O$ with
$\O_i$ being an open polygonal domain. It is well-known \cite{Kel:74} that problem (\ref{scalar}) has a unique solution
$u$ in $H^{1+s}(\O)$ with possibly very small $s>0$.

Recently, in \cite{CaZh:15apriori} we obtained a priori error estimates for the nonconforming, the mixed, and the discontinuous 
finite element approximations. Those estimates are robust with respect to the diffusion coefficient and optimal with respect to
local regularity of the solution with no assumption on the distribution of the diffusion coefficient. 
This paper is a continuation of \cite{CaZh:15apriori} on the a posteriori error estimates. We refer readers to \cite{CaZh:15apriori}
for definitions of commonly used notations.

For the conforming finite element approximation  to the interface problem in (\ref{scalar}), 
by using the diffusion coefficient to properly weight
the element residual and the edge flux jump, Bernardi and Verf\"urth in \cite{BeVe:00} (see also \cite{Pet:02}) 
showed that the resulting residual
based error estimator is locally efficient and globally reliable with the efficiency constant independent of 
the jump of the diffusion coefficient.
Moreover, under the assumption of the quasi-monotone distribution of the diffusion coefficient 
(see section 1.1 of \cite{CaZh:15apriori} on the QMA), 
the reliability constant is proved to be uniform with respect to the jump as well. 
Since then, various robust a posteriori error estimators have been constructed, analyzed, and implemented
(see, e.g., \cite{LuWo:04, CaZh:08a, Voh:11} for the conforming elements, \cite{Ain:05, CaZh:10a} 
for the nonconforming elements,
\cite{Ain:07, Kim:07, CaZh:10a, DuLinZhang:16} for the mixed elements,  \cite{CaYeZh:09,ErSt:08,ErStZu:08} for the discontinuous elements, \cite{MuJa:13,MuJa:14} for the finite volume methods).
The robustness for those estimators was theoretically established again under the QMA. 
However, numerical results by many researchers including ours strongly suggest that those estimators are robust 
even when the diffusion coefficients are not quasi-monotone. 

The purpose of this paper is to theoretically establish robust reliability bounds of the residual estimators 
without the QMA for nonconforming and discontinuous finite element approximations. 
The QMA is imposed to guarantee the desired approximation and stability properties of the Cl\'ement type interpolation (see  \cite{BeVe:00} for details), which is one of the key steps in obtaining the reliability bound of the residual based estimator. 

For the conforming elements, this type of interpolations is defined on vertex patches through averaging and, hence, the QMA is required. One reason of using the Cl\'ement interpolation is due to its minimum regularity requirement of the approximated function as in the commonly used reliability analysis (see, e.g., \cite{BeVe:00}).
For the nonconforming and discontinuous elements, one may construct a modified 
Cl\'ement interpolation satisfying the desired properties without the QMA (see \cite{CaHeZh:14}).
Due to the lack of the error equation, the reliability bound for discontinuous approximations is commonly
analyzed through the Helmholtz decomposition of the true error. Application of the Helmholtz decomposition, in turn,
leads to establishment of the reliability bound for conforming approximations and, hence, the requirement of the QMA.

In \cite{CaHeZh:14}, we introduced a new and direct analysis, that does not involve the Helmholtz decomposition,
for the two-dimensional nonconforming elements. In particular, we derived an $L^2$ representation of the error in the energy norm 
that naturally contains three terms: the element residual, the face flux jump, and the face solution jump. 
Due to a technical difficulty, the solution jump was modified at elements where the QMA is not satisfied. 
The modified estimator was proved to be robustly reliable without the QMA. Unfortunately, robustness of local efficiency of the modified indicator requires the QMA.

With the help of our newly developed trace inequality in \cite{CaZh:15apriori}, 
we are able to bound the solution jump without any modification.
Moreover, instead of using the nonconforming Cl\'ement type interpolation as in \cite{CaHeZh:14}, we use the standard nonconforming interpolation and the piecewise constant projection for the respective nonconforming 
and discontinuous elements, which fully takes advantage of the local feature of the element itself.  
Both the approximations are element-wisely defined. 
Thus, without the QMA, we are able to prove the robustness in both the
two- and three-dimension for the nonconforming and discontinuous elements.

The paper is organized as follows. Section~2 describes the nonconforming and the discontinuous Galerkin
finite element approximations and some preliminaries. Residual based a posteriori error estimators are described in
section~3, and local efficiency bounds are stated in section~4. Without the QMA, the robust reliability bounds are proved 
for the nonconforming and the discontinuous Galerkin elements in both the two- and three-dimensions
in the section 5. \\

\section{Discontinuous Finite Element Approximations and Preliminaries}
\setcounter{equation}{0}

Assume that the triangulation $\cT$ is regular 
and that the physical interfaces $\{\p\O_i\cap\p\O_j\,:\, i,j=1,\,...,\,n\}$ do not cut
through any element $K\in\cT$. 
For the nonconforming finite element approximation, we only consider the Crouzeix-Raviart (CR)
linear element, and denote the  CR nonconforming linear finite element space by
$$
V^{cr} := \{ v \in L^2(\O) : v|_K \in P_1(K) \; \,\forall \, K\in\cT\,
  \mbox{ and }\,\int_F \jump{v}\,ds =0 \; \forall  \,F\in\cE\},
$$
where $\cE$ is the set of all faces of the triangulation $\cT$. 
(A face in this paper means an edge or a face in the respective two- or three-dimension.)
Denote the discontinuous finite element space of degree $k\ge 0$ by
$$
D_k := \{ v \in L^2(\O) : v|_K \in P_k(K) \; \,\forall \, K\in\cT\}.
$$

The corresponding variational formulation of problem (\ref{scalar}) with homogeneous boundary conditions
 is to find $u\in H_{0}^1(\O)$ such that
 \begin{eqnarray}\label{galerkin}
(\a \nabla u,\,\nabla v) &=& (f,v)\qquad \,\forall\,v\in H_{0}^1(\O).
\end{eqnarray}
The CR nonconforming finite element approximation is to find $u^{cr} \in V^{cr}$ such that
\begin{eqnarray} \label{problem_nc}
  ( \a\nabla_h u^{cr} ,\, \nabla_h v) &=& (f,v)  \qquad \,\forall\, v\in
  V^{cr}.
\end{eqnarray}
For any $K\in\cT$ and some $\alpha>0$, let 
 \[
 V^{1+\alpha}(K)=\{v\in H^{1+\a}(K)\,:\, \Delta\, v \in L^2(K)\}
 \]
and let 
 \[
  V^{1+\a}(\cT) :=\{v\, :\, v|_K\in V^{1+\alpha}(K)\,\, \forall\, K\in\cT\}.
  \]
In \cite{CaYeZh:09} we introduced the following variational formulation for the interface 
problem in (\ref{scalar}): find $u\in V^{1+\epsilon}(\cT)$ with $\epsilon >0$
 such that
 \beq\label{DGV}
 a_{dg}(u,\,v) = (f,\,v)\quad\forall\,\, v \in V^{1+\epsilon}(\cT),
\eeq
where the bilinear form $a_{dg}(\cdot,\,\cdot)$ is given by
\begin{eqnarray*}
a_{dg}(u,v)&=&(\a\nabla_h u,\nabla_h v)
 +\sum_{F\in\cE}\int_F\gamma  \dfrac{\aH}{h_F} \jump{u}
 \jump{v}\,ds \\[2mm] \nonumber
&&\quad -\sum_{F\in\cE}\int_F\{\a\nabla
u\cdot\bn_F\}_{w}^F \jump{v}ds -
\sum_{F\in\cE}\int_F \{\a\nabla
v\cdot\bn_F\}_{w}^F\jump{u}ds.
\end{eqnarray*}
Here, $\{\cdot\}_w^F$ is the weighted average,
$\aH$ is the harmonic average of $\a$ over $F$,
and $\gamma$ is a positive constant only depending on the shape of elements.
The discontinuous Galerkin finite element method  is then to
seek $u^{dg}_k \in D_k$ such that
\beq\label{problem_dg}
a_{dg}(u^{dg}_k,\, v) = (f,\,v)\quad \forall\, v\in D_k.
\eeq
Difference between (\ref{DGV}) and (\ref{problem_dg}) leads to the following error equation:
\beq \label{erroreq_dg}
a_{dg}(u-u^{dg}_k,\,v) = 0 \quad \forall \, v\in D_k.
\eeq
For simplicity, we consider only this symmetric version of the interior penalty discontinuous Galerkin finite element method 
since its extension to other versions of discontinuous Galerkin approximations is straightforward.
Define the jump semi-norm and the DG norm by 
$$
\|v\|_{J,F} = \sqrt{\dfrac{\aH}{h_F}} \| \jump{v} \|_{0,F}
\quad\mbox{and}\quad
\tri v\tri_{dg} = \left(\|\a^{1/2}\nabla_h v\|_{0,\O}^2+\sum_{F\in \cE}\|v\|_{J,F} ^2 \right)^{1/2},
$$
respectively, for all $v\in  H^1(\cT)$. 

In order to guarantee the robustness of the error estimate with respect to $\a$, 
we choose harmonic weights in this paper:
$\omega_F^\pm = \dfrac{\a_F^\mp}{\a_F^-+\a_F^+}$.
It is easy to show that
\begin{equation} \label{weight1}
w_F^{\pm}\a_F^{\pm} \le \sqrt{\a^{\pm}\aH}, \quad\,\,
 \frac{\omega_F^+}{\sqrt{\alpha_F^-}} \le \sqrt{\frac{1}{\alpha_{F,A}}},
	\,\quad \mbox{and} \quad\,
	\frac{\omega_F^-}{\sqrt{\alpha_F^+}} \le \sqrt{\frac{1}{\alpha_{F,A}}}.
\end{equation}
We shall use the following commonly used identity:
 \begin{equation}\label{jump-id}
 \jump{ u v}_F = \{v\}^w_F\, \jump{u}_F + \{u\}_w^F\, \jump{v}_F.
 \eeq

An inequality proved in \cite{CaZh:15apriori} plays an important role to bound the solution jump
in the reliability analysis for the nonconforming and discontinuous elements. For the convenience of readers,
we cite it here.

\begin{lem} 
Let $F$ be a face of $K\in\cT$, $\bn_F$ the unit vector normal to $F$, and $s>0$.
Assume that $v$ is a given function in $V^{1+s}(K)$.
For any $w_h\in P_k(K)$, we have 
 \begin{eqnarray}\label{tracecombined}  
 \int_F \left(\nabla v\cdot\bn_F\right)\, w_h \,ds 
  &\leq & C\, h_F^{-1/2}\|w_h\|_{0,F}
 \left(\|\nabla v\|_{0,K} + h_K\|\Delta v\|_{0,K}\right),   
\end{eqnarray}
with the constant $C$ independent of $h$ and $v$.
\end{lem}

\section{Residual-Based A Posteriori Error Estimators}
\setcounter{equation}{0}

This section describes local indicators and global estimators for nonconforming and discontinuous Galerkin finite element approximations.
The estimators for the nonconforming 
elements introduced in \cite{CaHeZh:14} and  for the discontinuous elements in this paper are more accurate than the existing estimators
(see, e.g., \cite{Ain:05, CaYeZh:09})
and differ in replacing the face tangential
derivative jumps by the face solution jumps.

 \subsection{CR Nonconforming Elements}
 
For each $F\in\cE$, 
let $\bt_F$ be the unit vector tangent to $F$ for $d=2$, and $\bn_F$ be
the unit vector normal to $F$. Denote the tangential component of
a vector field $\btau$ on $F$ by
 \[ 
 \gamma_F(\btau) := \left\{ \begin{array}{llll} \btau\cdot\bt_F, & d=2,\\[2mm]
 \btau\times \bn_F, & d=3.
 \end{array}
 \right.
 \]
Denote the element residuals and the corresponding indicators by  
 \[
  r^{cr}_K=f_0 
  \quad \mbox{and} \quad
  \eta_{r,K}^{cr} = \dfrac{h_\sK}{\sqrt{\a_\sK}}\|r_K^{cr}\|_{0,K}, \quad \forall \; K\in \cT,
  \]
 respectively. 
 Denote the respective face flux and tangential derivative jumps by
 \[
 j^{cr}_{n,F}=\jump{\a \nabla_h u^{cr}\cdot\bn}_F, \quad\forall \,\, F \in \cE_I
 \quad\mbox{and}\quad
  j^{cr}_{t,F}=\jump{\gamma_t(\nabla_h u^{cr})}_F, \quad\forall\,\, F\in\cE,
 \]
and the indicators corresponding to the face flux, tangential 
 derivative, and solution jumps by
  \[
   \eta_{j,n,F}^{cr}=\sqrt{\dfrac{h_F}{\aA}}\, \|j_{n,F}^{cr}\|_{0,F}, \quad
   \eta_{j,t,F}^{cr}=\sqrt{\aH h_F} \,\|j_{t,F}^{cr}\|_{0,F},
  \quad\mbox{and}\quad
  \eta_{j,u,F}^{cr} = \sqrt{ \dfrac{\aH}{h_F}}\, \| \jump{u^{cr}}\|_{0,F},
  \]
 respectively.
 Then the local indicator of the residual type for the nonconforming elements, introduced in 
 \cite{CaHeZh:14} and to be studied in this paper, is given by
\[
	\eta_K^{cr} =
	\left(
	 	 \left({\eta^{cr}_{r,K}} \right)^2
	 	 + \sum_{F\in \cE_K \cap \cE_I }      
		 	                \dfrac{1}{2} \left( \eta_{j,n,F}^{cr} \right)^2   
			                 + \sum_{F\in  \cE_K \cap \cE_I}
		 	\dfrac{1}{2} \left( {\eta^{cr}_{j,u, F}}\right)^2 
			  +\sum_{F\in  \cE_K \cap \cE_D}
		 	 \left( {\eta^{cr}_{j,u, F}}\right)^2	
	 \right)^{1/2}.
\]
Now the global estimator for the nonconforming elements is given by
 \[
 \eta_{cr}= 
 \left( \sum_{K \in \cT}
  \left(\eta_K^{cr} \right)^2 \right)^{1/2}
 =  \left(  \sum_{K\in\cT} \left( \eta_{r,K}^{cr} \right)^2 
  + \sum_{F\in\cE_I}   \left( \eta^{cr}_{j, n,F} \right)^2 
  + \sum_{F\in\cE}  \left(\eta^{cr}_{j,u, F}\right)^2
  \right)^{1/2}.
 \]
 
\begin{lem}\label{lem:svt}
Let $F\in\cE_I$.
For any $v^{cr} \in V^{cr}$, we have 
\beq \label{crjump}
 \|\jump{v^{cr}}\|_{0,F} \,\,
  \left\{\begin{array}{ll}
  =\dfrac{1}{\sqrt{12}}\,h_F\, \|\jump{\nabla v^{cr} \cdot \bt_F}\|_{0,F}, & \mbox{if  } d=2,\\[6mm]
  \leq C\, h_F\, \|\jump{\nabla v^{cr}\times \bn_F}\|_{0,F}, & \mbox{if  } d=3,
  \end{array}\right.
\eeq
where $C$ is a positive constant independent of $h_K$.
\end{lem}

Here and thereafter, we use $C$ with or without subscripts in this
paper to denote a generic positive constant, possibly different at
different occurrences, that is independent of the mesh parameter
$h_K$ and the ratio $a_{\max}/a_{\min}$ but may depend on the domain
$\O$.

\begin{proof}
The equality in (\ref{crjump}) is proved in \cite{CaHeZh:14} by a direct calculation.
To show the validity of the inequality in (\ref{crjump}), for any $v^{cr} \in V^{cr}$, note first that 
its jump $\jump{v^{cr}}_F$ over $F\in\cE_I$ is a linear function and that $\int_F \jump{v^{cr}}\,ds =0$.  
If $\nabla \jump{v^{cr}}_F \times \bn_F$ vanishes on $F$, then the jump $\jump{v^{cr}}_F$ on $F$ is constant and, hence, zero. 
Now, the inequality in (\ref{crjump}) follows from the norms equivalence in a finite dimensional space
and the standard scaling argument. 
\end{proof}

\begin{rem}
 Instead of the face solution jumps, 
 existing residual based error estimators for the nonconforming elements 
 use the face tangential derivative jumps. {\em Lemma~\ref{lem:svt}} indicates that 
 our estimator $\eta_{cr}$ is less than the existing estimator and, hence, it is more accurate
 {\em (}see {\em Figure 6} in {\em \cite{CaHeZh:14})}.
 \end{rem}

\subsection{Discontinuous Elements}

Denote the element residuals and the corresponding indicators by  
 \[
  r^{dg}_K=f_{k-1}+\gradt (\a\nabla u^{dg}_k)  
  \quad \mbox{and} \quad
  \eta_{r,K}^{dg} = \dfrac{h_\sK}{\sqrt{\a_\sK}}\|r_K^{dg}\|_{0,K},
  \quad \forall \, K \in \cT.
  \]
 respectively.
 Denote the respective face flux and solution jumps by
 \[
   j^{dg}_{n,F}=\jump{\a \nabla_h u^{dg}_k\cdot\bn}_F, \,\;\forall \,\, F \in \cE_I
   \quad\mbox{and}\quad
 j^{dg}_{u,F}=\jump{u_k^{dg}}, \,\, \forall \,\, F\in \cE,
 \]
and the indicators corresponding to the face flux and solution jumps by
  \[
   \eta_{j,n,F}^{dg}=\sqrt{\dfrac{h_F}{\aA}}\, \|j_{n,F}^{dg}\|_{0,F}  
   \quad\mbox{and}\quad
  \eta_{j,u,F}^{dg} = \sqrt{ \dfrac{\aH}{h_F}}\, \| \jump{u^{dg}}\|_{0,F},
  \]
 respectively.
 Then the local indicator of the residual type for the discontinuous elements is given 
\cite{CaYeZh:09} by
\[
	\eta_K^{dg} =
	\left(
	 	 \left({\eta^{dg}_{r,K}} \right)^2
	 	 + \sum_{F\in \cE_K \cap \cE_I }      
		 	                 \dfrac{1}{2}\left( \eta_{j,n,F}^{dg} \right)^2   
			                 + \sum_{F\in  \cE_K\cap \cE_I} \dfrac{1}{2}
		 	 \left( {\eta^{dg}_{j,u, F}}\right)^2 	
			 +\sum_{F\in  \cE_K\cap \cE_D}
		 	 \left( {\eta^{dg}_{j,u, F}}\right)^2 
	 \right)^{1/2}.
\]
Now the global estimator for the nonconforming elements is given by
 \[
 \eta_{dg}= \left( \sum_{K \in \cT} \left(\eta_K^{dg} \right)^2 \right)^{1/2}
 = \left(  \sum_{K \in \cT}(\eta_{r,K}^{dg})^2 
  + \sum_{F\in\cE_I}  (\eta^{dg}_{j, n,F})^2 + \sum_{F\in\cE}  (\eta^{dg}_{j,u, F})^2
  \right)^{1/2}.
 \]

\section{Efficiency Bounds}
\setcounter{equation}{0}

Let $f_{k}$ be the $L^2$ projection of $f$ onto $D_k$ for $k\geq 1$. 
Denote local and global weighted oscillations by 
$$
{\mbox{osc}}_\a(f,K) = \dfrac{h_K}{\sqrt{\a_K}}  \|f-f_{k-1}\|_{0,K} \quad\mbox{and}\quad
{\mbox{osc}}_\a(f,\cT) = \left( \sum_{K\in\cT} {\mbox{osc}}_\a(f,K)^2 \right)^{1/2},
$$ 
respectively.
The local efficiency bounds of all indicators described in the previous section have been established
without the QMA. For example, for the discontinuous Galerkin finite element approximation, it is proved in \cite{CaYeZh:09} 
that for any $K\in \cT$, there exists a positive constant $C$ independent of $\a$ and $h_K$ such that
 \begin{equation} \label{local efficiency}
 \eta_K^{dg} \le C\, \left(\tri u-u_{dg}^k \tri_{U\triangle_K} + \osc_\a(f, \triangle_K)\right),
 \end{equation}
where $\triangle_K$ is a local neighborhood of $K$. 

The key idea of the proof is to use either element or edge bubble functions
in order to localize the error as well as to simplify the boundary conditions. 
The proof of local efficiency bound similar to (\ref{local efficiency}) can be found in \cite{CaZh:10a, CaHeZh:14}
for the CR nonconforming element.

\section{Reliability Bounds for CR Nonconforming and Discontinuous Elements}
\setcounter{equation}{0}

In this section, we establish robust reliability bounds for the CR nonconforming and
discontinuous Galerkin approximations in both the two- and three-dimension.

\subsection{CR Nonconforming Elements}

Without the QMA, the robust reliability bound for the nonconforming elements was first established 
in \cite{CaHeZh:14} for a slightly modification of the estimator $\eta_{cr}$ in the two-dimension. 
The modification is due to the failure of bounding the solution jump term. 
This difficulty may be overcome by using the trace
inequality introduced in \cite{CaZh:15apriori}.

Instead of using the modified Cl\'{e}ment-type interpolation \cite{CaHeZh:14} which may be extended 
to the three-dimension in a straightforward manner, we use the standard nonconforming interpolation that can also be naturally extended to the three-dimension. Moreover, it is local
and, hence, has the approximation and stability properties needed for obtaining robust reliability bound
both without the QMA.
 To this end, let
 \begin{eqnarray*}
  W^{1,1}(\cT) &=&\{v\in L^2(\O)\, : \, v|_K\in W^{1,1}(K)\,\,\,\forall\, K\in \cT\}\\[2mm]
 \mbox{and}\quad 
 W (\cT) &=& \{v\in W^{1,1}(\cT)\, : \, \int_{F} \jump{v} ds =0\,\,\, \forall\, F\in \cE \}.
 \end{eqnarray*}
Denote by $\theta_F (\bx)$ the nodal basis function of $V^{cr}_\cT$ associated with the face
$F\in\cE$, i.e., 
 \[
 \dfrac{1}{|F^\prime|} \int_{F^\prime} \theta_F (\bx)\, ds
 =\delta_{FF^\prime}
 \,\,\,\forall\, F^\prime\in\cE,
 \]
where $\delta_{FF^\prime}$ is the Kronecker delta. The local and global Crouzeix-Raviart interpolations are defined respectively by
 \[
  I^{cr}_K v = \sum_{F\in \cE_K} \left(\dfrac{1}{|F|} \int_{F} v ds \right) \theta_F(\bx)
 \quad\mbox{and}\quad
  I^{cr} v = \sum_{F\in \cE } \left(\dfrac{1}{|F|} \int_{F} v ds \right) \theta_F(\bx)
 \]
for the respective $v\in W^{1,1}(K)$ and $v\in W (\cT)$.
It was shown (see, e.g., Theorem 1.103 and Example 1.106 (ii) of \cite{ErGu:04}) that
for $v\in H^{1}(K)$ 
\beq \label{localcr}
\|v-I^{cr} v\|_{0,K} \leq C\, h_{K}\, \|\nabla v\|_{0,K} \quad\mbox{and}\quad
\| \grad(v-I^{cr} v)\|_{0,K} \leq C\, \|\nabla v\|_{0,K}.
\eeq

 \begin{thm}\label{th:relcr}
Let $u$ and $u^{cr}$ be the solutions of {\em (\ref{galerkin})} and {\em (\ref{problem_nc})}, respectively. 
Without the QMA in both the two- and three-dimension, the estimator $\eta_{cr}$ for the nonconforming elements satisfies 
the following robust reliability bound:
 \beq \label{relcr}
 \|\a^{1/2}\nabla_h (u-u^{cr})\|_0
  \leq C\,\left(\eta_{cr} + {\mbox{osc}}_\a (f)\right),
 \eeq
 where $C$ is a positive constant independent of the $\a$.
\end{thm}

\begin{proof}
Let  
 \[
 e^{cr}=u-u^{cr}
 \quad\mbox{and}\quad
 e_I^{cr} = I^{cr}u - u^{cr} = I^{cr} e^{cr}.
 \]
Then we have the following $L^2$ representation of the true error in the (broken) energy norm 
(see Lemma 2.1 of {\cite{CaHeZh:14}}):
 \[
  \|\a^{1/2}\nabla_h e^{cr}\|_0^2
   = \sum_{K \in \cT} (f, e^{cr}-	
	e_I^{cr})_K -\sum_{F\in \cE_I} \int_F j_{n,F}^{cr} \,\{e^{cr}-e_I^{cr}\}^w\,ds
	 -\sum_{F \in \cE} \int_F \{ \alpha \nabla e^{cr} \cdot \bn \}_w \, \jump{u^{cr}} \,ds.
	 \]
The first two terms of the above equality 
may be bounded in a similar fashion as that in \cite{CaHeZh:14}. 
That is, it follows from the Cauchy-Schwarz and triangle or trace inequalities, (\ref{localcr}), 
and (\ref{weight1}) that
 \begin{equation}\label{CR-reliability 1}
  (f, \, e^{cr}-e_I^{cr})_K
	 \le  C\, \left( \eta^{cr}_{r,K} +\osc_\a(f,K)\right)\|\alpha^{1/2} \nabla_h e^{cr}\|_{0,K}, 
	\quad \forall \,\, K \in \cT, 
\end{equation}
and
\begin{eqnarray}
	\int_F j_{n,F}^{cr} \{e- e_I^{cr} \}^w \,ds 
	&\le& 
		C\| j_{n,F}^{cr}\|_{0,F}  
		\{\omega^+\|(e^{cr}-e_I^{cr})|_{K_F^-}\|_{0,F}
		+\omega^-\|(e^{cr}-e_I^{cr})_{K_F^+}\|_{0,F} \} \nonumber
	\\[2mm] \label{CR-reliability 2}
	&\le& C \sqrt{\frac{h_F}{\aA}} \| j_{n,F}^{cr}\|_{0,F} \, \|\alpha^{1/2} \nabla e^{cr}\|
	_{K_F^+ \cup K_F^-},
	\quad \forall F \,\, \in \cE_I.
\end{eqnarray}
To bound the third term on the solution jump, the key is the inequality in (\ref{tracecombined}), which
together with (\ref{weight1}) and the local efficiency bound of the element residual,
yields
\begin{eqnarray} \nonumber
		 && \int_F\{\alpha \nabla e^{cr} \cdot \bn_F\}_w\, \jump{u^{cr}} \,ds \\[2mm] \nonumber
		&=&\int_F \jump{u^{cr}} \,\left(\o^+( \alpha^+ \nabla e^{cr}\cdot \bn)|_{K^+} 
			+\o^- (\alpha^- \nabla e^{cr}\cdot \bn)|_{K^-} \right)\,ds\
		\nonumber \\[2mm]
		&\le &C\, \sqrt{\frac{ \alpha_{F,H}}{h_F}} \,\, \| \jump{u^{cr}}\|_{0,F} \, 
			\sum_{K \in \cT_F} \left(\|\alpha^{1/2} \nabla e^{cr}\|_{0,K}+
				h_K \alpha_K^{-1/2}\,\|f+\nabla \cdot (\a \nabla u^{cr})\|_{0,K}\right)
		\nonumber \\[2mm]  \label{CR-reliability 3}
		&\leq & C\, \eta^{cr}_{j,u,F} \,\sum_{K \in \cT_F}\left(\|\alpha^{1/2} \nabla_h e^{cr}\|_{0,K}		
		 +{\mbox{osc}}_{\alpha}(f,K) \right), \quad \forall \,\,F \in \cE.
\end{eqnarray}
Summing (\ref{CR-reliability 1}) over $K\in \cT$,  (\ref{CR-reliability 2}) over $F\in \cE_I$,  and 
(\ref{CR-reliability 3}) over $F\in \cE$ implies the validity of (\ref{relcr}). This completes the proof of
the theorem.
\end{proof}

\subsection{Discontinuous Elements}

Without the QMA, the robust reliability bound for the discontinuous elements may be obtained in a similar
fashion as that for the nonconforming elements. Again, the key steps are the $L^2$ representation of
the true error and the inequality (\ref{tracecombined}) to bound the solution jump. Moreover, we simply use
the local constant average of the error instead of the modified Cl\'{e}ment interpolation due to the complete
local feature of the discontinuous elements.

Let $u$ and $u_k^{dg}$ be the solutions of (\ref{galerkin}) and (\ref{problem_dg}), respectively. Denote 
the true error by
 \[
  e^{dg} = u-u^{dg}_k. 
   \]
Let $\bar{e}^{dg}$ be piecewise constants on $\cT$ with $\bar{e}^{dg}|_K$ being 
the average of $e^{dg}$ on $K\in\cT$. It is well known that 
 \beq \label{localdg}
\|e^{dg} - \bar{e}^{dg}\|_{0,K} \leq C\, h_{K}\, \|\nabla\, e^{dg}\|_{0,K}, \quad \forall \,K\in \cT, 
\eeq
where $C$ only depends on the regularity of $\cT$.

\begin{lem}\label{5.6}
The true error of the discontinuous finite element approximation in the broken energy norm has the 
following error representation:
\begin{eqnarray}  \nonumber
 \|\a^{1/2} \nabla_h \, e^{dg}\|_0^2
 &=& \sum_{K\in\cT}  \left(f+\nabla \cdot (\a \nabla u_k^{dg}), \, e^{dg}-\bar{e}^{dg} \right)_K 
  - \sum_{F\in\cE} \int_F \{\a\grad e^{dg}\cdot\bn \}_w\, \jump{u_k^{dg}}\, ds \\ [2mm] \label{dg:5}
 &-&  \sum_{F\in\cE_I} \int_F \jump{\a\grad u^{dg}\cdot\bn} \,\{ e^{dg} - \bar{e}^{dg}\}^w\, ds 
   -\sum_{F\in\cE} \int_F \gamma \dfrac{\aH}{h_F}\, \jump{u_k^{dg}} \,\jump{\bar{e}^{dg}}\,ds. \,\,\,\,\,\,
\end{eqnarray}
\end{lem}

\begin{proof}
It follows from the error equation in (\ref{erroreq_dg}), integrations by parts, (\ref{jump-id}), 
the continuities of the solution $u$ and the 
normal component of the flux $-\a\nabla\,u$ across any face $F\in \cE_I$, and the homogeneous 
Dirichlet boundary condition that
\begin{eqnarray*}
  && \|\a^{1/2} \nabla_h \, e^{dg}\|^2_0
   = \left(\a \grad_h e^{dg}, \,\grad_h (e^{dg} -\bar{e}^{dg} )\right)  \\ [2mm]
 &=& \sum_{K\in\cT} \left(f+\gradt (\a\nabla u^{dg}_k), \, e^{dg}-\bar{e}^{dg} \right)_K 
   +\int_{\p K} (\a\grad e^{dg}\cdot\bn)(e^{dg} - \bar{e}^{dg}) \, ds \\ [2mm]
 &=&  \sum_{K\in\cT} \left(f+\gradt (\a\nabla u^{dg}_k), \, e^{dg}-\bar{e}^{dg} \right)_K 
   - \sum_{F\in\cE_I} \int_F \{\a\grad e^{dg}\cdot\bn \}_w\, \jump{u_k^{dg} + \bar{e}^{dg}} \,ds \\[2mm]
 &&-\sum_{F\in\cE_I} \int_F \jump{\a\grad u^{dg}\cdot\bn} \, \{ e^{dg} - \bar{e}^{dg}\}^w ds 
  - \sum_{F\in\cE_D} \int_F(\a\grad e^{dg}\cdot\bn)(u_k^{dg} + \bar{e}^{dg}) \,ds.
\end{eqnarray*}
On the other hand, the fact that $a_{dg}(e^{dg}_k,\, \bar{e}^{dg}) = 0$ implies
 \[ 
 \sum_{F\in\cE}\int_F\gamma  \dfrac{\aH}{h_F} \,\jump{e^{dg}} \,\jump{\bar{e}^{dg}}\,ds  
  -\sum_{F\in\cE_I}\int_F\{\a\nabla e^{dg}\cdot\bn\}_{w}\, \jump{\bar{e}^{dg}}\,ds   
   -\sum_{F\in\cE_D}\int_F(\a\nabla e^{dg}\cdot\bn ) \,\bar{e}^{dg}\,ds
   =0.
 \] 
Combining the above two equalities gives (\ref{dg:5}). This completes the proof of the lemma.
\end{proof}

\begin{thm} 
Let $u$ and $u_k^{dg}$ be the solution of (\ref{galerkin}) and (\ref{problem_dg}), respectively.
Without the QMA in both the two- and three-dimension, the estimator $\eta_{dg}$ for the discontinuous 
element approximation satisfies the following robust reliability bound:
\beq \label{reldg}
 \tri u-u^{dg}_k\tri_{dg} 
 \leq C\,\left(\eta_{dg} + {\mbox{osc}}_\a(f,\,\cT)\right),
\eeq
where $C$ is a positive constant independent of the $\a$.
\end{thm}

\begin{proof}
By the definition of the DG norm $ \tri \cdot\tri_{dg} $, to prove the validity of (\ref{reldg}), it suffices to show that
 \beq \label{reldg1}
 \|\a^{1/2}\grad_h\, (u-u^{dg}_k)\|_0
 \leq C\,\left(\eta_{dg} + {\mbox{osc}}_\a(f,\,\cT)\right).
\eeq
To this end, denote the estimators corresponding to the element residual, the face flux jump, and the 
 solution jump by
  \[
   \eta^{dg}_r = \left(\sum_{K\in\cT} \left(\eta^{dg}_{r,K}\right)^2 \right)^{1/2},
   \quad 
   \eta^{dg}_{j,n} =\left(\sum_{F\in\cE_I} \left(\eta_{j,n,F}^{dg} \right)^2\right)^{1/2},
   \quad\mbox{and}\quad
   \eta^{dg}_{j,u} = \left( \sum_{F\in\cE} \left( \eta_{j,u,F}^{dg}\right)^2 \right)^{1/2}.
   \]
  respectively. Denote four terms in Lemma~\ref{5.6} by $I_1$, $I_2$, $I_3$, and $I_4$, respectively.
Hence,
 \[
  \|\a^{1/2} \nabla_h \, (u-u^{dg}_k)\|_0^2 = I_1+I_2+I_3+I_4.
  \]
 In a similar fashion as the proof of Theorem~\ref{th:relcr}, the $I_i$ for $i=1,\,2,\,3$
may be bounded as follows:
 \begin{eqnarray*} \label{dg1}
 I_1
  & \le & C\, \left( \eta^{dg}_r
   + \osc_{\a}(f,\cT) \right)  \|\a^{1/2}\grad_h \,e^{dg}\|_0, \quad 
   I_2  
	\leq  C\, \eta^{dg}_{j, u}
	\left( \|\a^{1/2} \grad_h\, e^{dg} \|_0 +  \osc_{\a}(f, \cT)\right),
		\\[2mm] \label{dg3}
  \mbox{and}\quad  I_3
   & \le&   C\,  \eta^{dg}_{j,n}\, 
    \|\a^{1/2}\grad_h\, e^{dg}\|_0.
	\end{eqnarray*}

To bound $I_4$, for all $F\in\cE$, it follows from the triangle and the Cauchy-Schwartz inequalities, 
the continuity of the solution $u$, the trace inequality, and (\ref{localdg}) that 
 \begin{eqnarray*}
   && \int_F \dfrac{\aH}{h_F} \, \jump{u^{dg}_k}\, \jump{\bar{e}^{dg}} \, ds 
   \leq \int_F \dfrac{\aH}{h_F} \, \jump{u^{dg}_k}\, 
      \left(\|\jump{e^{dg}}\|_{0,F} + \|\jump{e^{dg}-\bar{e}^{dg}}\|_{0,F}\right)\,ds
   \\ [2mm]
  &\leq  & 
  \left(\eta_{j,u,F}^{dg}\right)^2 
   + \eta_{j,u,F}^{dg}\,\sqrt{\dfrac{\aH}{h_F}}\, \|\jump{e^{dg}-\bar{e}^{dg}}\|_{0,F}
   \leq \left(\eta_{j,u,F}^{dg}\right)^2 
   + C\, \eta_{j,u,F}^{dg}\,  \sum_{K \in \cT_F} \|\a^{1/2}\grad_h e^{dg}\|_{0, K}.
\end{eqnarray*}
Summing over all faces $F \in \cE$ and using the Cauchy-Schwarz inequality give
 \[ 
  I_4= \gamma\,\sum_{F \in \cE} \int_F \dfrac{\aH}{h_F}\, \jump{u^{dg}_k}\, \jump{\bar{e}^{dg}}\, ds  
  \leq  C\, \left(  \left(\eta^{dg}_{j,u}\right)^2 +  \eta^{dg}_{j,u}\, \|\a^{1/2} \grad_h\, e^{dg} \|_0 \right).
\] 
Combining the bounds for all $I_i$ and using the Cauchy-Schwarz inequality, we have
 \[
 \|\a^{1/2} \nabla_h \, e^{dg}\|_0^2
 \leq C\, \left( \eta^{dg}_{r} + \eta_{j,n}^{dg}    +  \eta_{j,u}^{dg} 
  + \osc_{\a}(f, \cT)\right) \, \|\a^{1/2} \grad_h\, e^{dg} \|_0 
  + C \left( \left( \eta_{j,u}^{dg} \right)^2 + \osc_{\a}(f,\cT)^2 \right),
 \]
which, together with the Cauchy-Schwarz inequality, implies 
 \[
 \|\a^{1/2} \nabla_h \, e^{dg}\|_0 \leq C\,\left( \eta_{dg} + \osc_{\a}(f,\cT)\right).
 \]
This proves the validity of (\ref{reldg1}) and, hence, the theorem.
This completes the proof of the theorem.
\end{proof}


\begin{thebibliography}{99}

\bibitem{Ain:05}  {\sc M. Ainsworth}, {\em Robust a posteriori error estimation for nonconforming finite element approximation}, SIAM J. Numer. Anal., 42 (2005), 2320--2341.

\bibitem{Ain:07}  {\sc M. Ainsworth}, {\em A posteriori error estimation for lowest order Raviart-Thomas mixed finite elements}, SIAM J. Sci. Comput., 30 (2007), 189--204.

%
%


%

 \bibitem{BeVe:00}
 {\sc C. Bernardi and R. Verf\"urth},
 {\em Adaptive finite element methods for elliptic equations with
 non-smooth coefficients},
 Numer. Math., 85:4 (2000), 579-608.
 
%
%

%

%



\bibitem{CaHeZh:14}
{\sc Z. Cai, C. He, and S. Zhang},
{\em Residaul-based a posteriori error estimate for interface problems: 
nonconforming linear elements}, Math. Comp., submitted.

\bibitem{CaYeZh:09}
 {\sc Z. Cai, X. Ye, and S. Zhang},
 {\em Discontinuous Galerkin finite element methods for interface problems:
 a priori and a posteriori error estimations},
SIAM J. Numer. Anal., 49:5 (2011), 1761--1787.
 
   \bibitem{CaZh:08a}
 {\sc Z. Cai and S. Zhang},
 {\em Recovery-based error estimator for interface problems:
 conforming linear elements},
 SIAM J. Numer. Anal., 47:3 (2009), 2132-2156.

 \bibitem{CaZh:10a}
 {\sc Z. Cai and S. Zhang},
 {\em Recovery-based error estimator for interface problems: mixed and nonconforming elements},
  SIAM J. Numer. Anal, 48:1 (2010), 30--52.
  
 
  
  
 
 
 \bibitem{CaZh:15apriori}
 {\sc Z. Cai and S. Zhang}, {\em 
Finite element methods for interface problems: Robust and local optimal a priori error estimates}, 
SIAM J. Numer. Anal., submitted, 2015.

  
%
  
%
%
%


\bibitem{DuLinZhang:16}
{\sc S. Du, R. Lin, and Z. Zhang},
{\em
A posteriori error analysis of multipoint flux mixed finite element method for interface problems}, Adv. Comp. Math., 2016.
%


%

\bibitem{ErGu:04}
{\sc A. Ern and J.-L. Guermond}, {\em Theory and Practice of Finite Elements}, vol. 159 of Applied Mathematical Series, Springer, New York, 2004.

 \bibitem{ErSt:08}
 {\sc A. Ern and A. Stephansen},
 {\em A posteriori energy-norm error estimates for
  Discontinuous Galerkin method with weighted averages for advection-diffusion
  equations approximated by weighted interiori penalty methods},
 J. Comput. Math., 26:4 (2008), 488-510.

 \bibitem{ErStZu:08}
 {\sc A. Ern, A. Stephansen, and P. Zunino},
 {\em A discontinuous Galerkin method with weighted averages for
 advection-diffusion equations with locally small and anisotropic diffusivity},
 IMA J. of Numer. Anal., 29:2 (2009), 235-256.

%

%
%
%
%
 
  \bibitem{Kel:74}
 {\sc R.B. Kellogg},
 {\em On the Poisson equation with intersecting interfaces},
 Appl. Anal., 4 (1975), 101-129.
 
  \bibitem{Kim:07}
 {\sc K-Y. Kim},
 {\em A posteriori error analysis for locally conservative mixed methods},
 Math. Comp., 76 (2007), 43-66.
 
  \bibitem{LuWo:04} {\sc R. Luce and B. I. Wohlmuth}, {\em A local a posteriori error estimator based on equilibrated fluxes}, SIAM J. Numer. Anal., 42:4 (2004), 1394--1414.
%
%
%
%

\bibitem{MuJa:13}
{\sc 
L. Mu and R. Jari}, {\em A recovery-based error estimator for finite volume methods of interface problems:
nonconforming linear elements}, Appl. Math. Comp., 220 (2013): 64-74.


\bibitem{MuJa:14}
{\sc L. Mu and R. Jari} {\em A posteriori error analysis for discontinuous finite volume methods of elliptic
interface problems}, J. Comp. Appl. Math., 255 (2014): 529-543.




  \bibitem{Pet:02}
 {\sc M. Petzoldt}, {\em A posteriori error estimators for elliptic
 equations with discontinuous coefficients},
 Adv. Comp. Math., 16:1 (2002), 47-75.
%
%
%
%
%


\bibitem{Voh:11}
{\sc M. Vohralk},
{\em Guaranteed and fully robust a posteriori error estimates for conforming discretizations of diffusion problems with discontinuous coefficients}, 
J. Sci. Comput., 46:3 (2011), 397--438

\end{thebibliography}
\end{document}